\documentclass[reqno]{amsart}
\usepackage{amssymb, amsmath, amsthm,latexsym,booktabs,todonotes, color, comment, lineno}

\newcommand{\mex}{{\bf mex}}
\newcommand{\Gr}{{\textnormal Gr}}

\def\N{\mathcal{N}}
\def\P{\mathcal{P}}

\theoremstyle{definition}
\newtheorem{definition}{Definition}

\theoremstyle{plain}
\newtheorem{lemma}[definition]{Lemma}
\newtheorem{proposition}[definition]{Proposition}
\newtheorem{theorem}[definition]{Theorem}

\usepackage{todonotes}

\begin{document}


\title{A classification of the Nim-sequences for the two-element subtraction game with a single pass}
\author{Richard M. Low}
\address{Department of Mathematics\\
         San Jose State University\\
         San Jose, CA 95192\\
         USA}
\email{richard.low@sjsu.edu}

\keywords{combinatorial game, Grundy-value sequence, periodicity}

\date{October 18, 2026. Version 0.1\\
\indent
\textit{2020 Mathematics Subject Classification.} 91A46}

\begin{abstract}
In this paper, we provide a complete description of the preperiods, periodic lengths, and the Grundy-value sequences of two-element subtraction games with a single pass option.
\end{abstract}
\maketitle

\section{Introduction} \label{S:intro}

Combinatorial game theory (CGT) was developed in the context of recreational mathematics. In their seminal work, Berlekamp, Conway and Guy \cite{WW2001} established the mathematical framework from which games of complete information could be studied. The power of this theory quickly became apparent and was utilized by many researchers (see Fraenkel's bibliography \cite{Fraen}). Along with its natural appeal, combinatorial game theory has applications to complexity theory, logic, and biology. 

Finite impartial games are an important class of combinatorial games. These games always terminate with a winner.  The set of available move options for both players are the same and perfect play is assumed. 

First, we recall some concepts (for finite impartial games) from CGT. Any terms which are not explicitly defined in this paper can be found in \cite{Alb}. Let $\mathbb{N}$ denote the set of nonnegative integers. For $S\subset {\mathbb{N}}$ , with $S\ne {\mathbb{N}}$, we define
mex$(S) = \min \{t\in {\mathbb{N}} : t\notin S\}$. If $G$ is a game position in
which there are no moves available, we set Gr$(G)=0$.
Then, if $G$ is any other position in the game and
$Q$ is the set of
positions that can be reached from $G$ in one move, then
\begin{equation*}
\textnormal{Gr}(G)=\textbf{mex}\{ \textnormal{Gr}(H) : H\in Q \}
\end{equation*}
and this is called the {\it Grundy-value} of $G$.
 A game $G$ that can be won by the player that made the
 move previous to $G$ is called a $\P$-\textit{position}. We say that $G$ is an $\N$-\textit{position} if
 $G$ can be won by the player to make the next move after $G$.
 Thus, a game $G$ is a $\P$-position
 if and only if \Gr$(G)=0$. \\

\noindent
 \textbf{Remark.} 
     Here are some important observations.
     \begin{itemize}
         \item A game $G$ is a $\P$-position if for each move
         $H$ from $G$, there is a move $K$ from $H$ such that
         $K$ is a $\P$-position.
         \item A game $G$ is an $\N$-position if there exists a
         move $H$ from $G$ such that $H$ is a $\P$-position.
     \end{itemize}

\vspace{10pt}

For decades, Guy and Nowakowski have  maintained a document of unsolved problems in CGT~\cite[pp. 465--466]{Guy}. 
Problem ``A1'' in this document is a long-standing open problem: characterize the structure of finite subtraction games. These games are played in the following way: From a single pile of $n$ stones, two players alternate removing $s_i$ stones, with a specified rule set $S = \{s_1, s_2, \dots, s_k\}$. If a player cannot move, they lose the game.

It is known that subtraction games with finite rule sets have periodic Grundy-value sequences. However, the periods of the games and the presence and structure of any pre-periods, remain largely unknown. In \cite{Austin, Flammenkamp}, some initial analysis on the Grundy-value sequences and their periodicity were made for subtraction games with rule set sizes two and three. Eventually, the case where $|S| = 2$ was completely determined in \cite{Ward_1}. Many Grundy-value sequences (for $|S| = k$) are given in~\cite[pp.~84--85, Table~1]{WW2001}. Additional results on the periodicity of Grundy-value sequences of subtraction games can be found in \cite{Ho, Lar, Ward2, Zhang}. For $S=\{x,y,z\}$ where $z=x+y$, the complete location of the P-positions (zeros) and the N-positions (ones, twos or threes) were described in \cite{Miklos_Post}. In \cite{Lar_Manabe}, Larsson and Manabe describe how to distinguish the N-positions having twos or threes in a restricted  setting. For every subtraction game, certain integers can be augmented to the set without altering the nim sequence of the set. A complete description of adjoinable integers to the set $S=\{x,y,x+y\}$ was given in \cite{Ho}.

In this paper, we study a variant of the subtraction game, which we call the \textit{Subtraction Game with a Single Pass} (denoted by subtraction set $S^*$). \\

\begin{itemize}
\item \textit{Subtraction Game with a Single Pass} is played like a normal subtraction game, with the option of a single pass which can be used by exactly one player. Once the pass option is used, it cannot be used again and the game continues in the usual fashion. The pass option can be used at any time, up to the penultimate move. It cannot be used at the end of the game. The player who cannot make a move loses the game.
\end{itemize}

\vspace{10pt}

\noindent
More specifically, we provide a complete description of the preperiods, periodic lengths, and the Grundy-value sequences of $S^\ast = \{x, y\}$.

\section{Preliminaries} \label{Prelim}

At the beginning of the game one pass is available to the game as a
whole.  Either player may use it, but after it is used it disappears.
A pass leaves the heap size unchanged.  A pass is not legal at a terminal
position. Since $x$ is the least subtraction, a position of size $n$ is terminal
exactly when $0\le n<x$.  Thus the pass is legal exactly when $n\ge x$. The phrase ``single pass'' throughout this paper means \emph{one pass
in total}, not one pass per player.

The availability of the pass produces two natural game states. Let $g(n)$ be the Sprague--Grundy value of the ordinary subtraction game, equivalently the value of the position after the unique pass has already been used. Let $h(n)$
be the Sprague--Grundy value of the position with heap size $n$ while
the unique pass remains available.

The ordinary values satisfy
\begin{equation}
g(n)=
\mex\Bigl(
   \{g(n-x):n\ge x\}
   \cup
   \{g(n-y):n\ge y\}
\Bigr).
\label{eq:g-rec}
\end{equation}

For the pass-available state,
\begin{equation}
h(n)=0,\qquad 0\le n<x,
\label{eq:h-terminal}
\end{equation}
and for $n\ge x$,
\begin{equation}
h(n)=
\mex\Bigl(
  \{h(n-x)\}
  \cup
  \{h(n-y):n\ge y\}
  \cup
  \{g(n)\}
\Bigr).
\label{eq:h-rec}
\end{equation}

The option $g(n)$ in \eqref{eq:h-rec} is the result of using the pass:
the heap does not change, but afterward the pass is no longer available.

Since \eqref{eq:h-rec} has at most three distinct options,
\[
        0\le h(n)\le3.
\]
The value $3$, which never occurs in the ordinary two-element game,
will turn out to occur only in a finite transient.

Throughout we write
\begin{equation}
        y=qx+r,\qquad q\ge1,\qquad 0\le r<x.
\label{eq:division}
\end{equation}

If $a$ is a symbol and $k\ge0$, the notation $a^k$ denotes $k$
consecutive copies of $a$.  Concatenation of words is written by
juxtaposition.\\

We require the explicit ordinary nim-word as an input. Theorem~\ref{thm:ordinary} is the two-element result found in \cite{Ward_1}.

\begin{theorem}[\cite{Ward_1} Ordinary two-element subtraction game]
\label{thm:ordinary}
Let $S=\{x,y\}$, $0<x<y$, and write $y=qx+r$ as in
\eqref{eq:division}.

\begin{enumerate}
\item If $q=2m$ is even, then one ordinary period is
\begin{equation}
        G=(0^x1^x)^m\,0^r2^{x-r}1^r.
\label{eq:ordinary-even}
\end{equation}
For $r=0$, this reduces to
\[
        G=(0^x1^x)^m2^x.
\]
Its least period is $x+y$.\\

\item If $q=2m+1$ is odd and $r>0$, then one ordinary period is
\begin{equation}
        G=(0^x1^x)^{m+1}2^r,
\label{eq:ordinary-odd-r}
\end{equation}
and its least period is $x+y$.\\

\item If $q$ is odd and $r=0$, then
\begin{equation}
        g=(0^x1^x)^\infty,
\label{eq:ordinary-odd-multiple}
\end{equation}
and its least period is $2x$.
\end{enumerate}
\end{theorem}

\section{Technical lemmas and propositions} \label{Lemmas}

Before the $y$-move becomes available, both $g$ and $h$ are especially
simple.

\begin{lemma}
\label{lem:pre-y}
Let
\[
        n=jx+t,\qquad 0\le t<x,
\]
and suppose $n<y$.

Then
\begin{equation}
        g(n)=j\bmod2.
\label{eq:g-before-y}
\end{equation}

For the pass-available game,
\begin{equation}
h(n)=
\begin{cases}
0,&j=0,\\
2,&j=1,\\
1,&j\ge2\text{ even},\\
0,&j\ge3\text{ odd}.
\end{cases}
\label{eq:h-before-y}
\end{equation}
\end{lemma}

\begin{proof}
For $n<y$, subtraction by $y$ is unavailable.  Therefore
\[
        g(n)=\mex\{g(n-x)\}
\]
whenever $n\ge x$.  Starting with $g(n)=0$ for $0\le n<x$, the
ordinary values alternate in blocks:
\[
        0^x1^x0^x1^x\cdots.
\]
This proves \eqref{eq:g-before-y}.

For $h$, while $n<y$ we have
\[
        h(n)=\mex\{h(n-x),g(n)\}.
\]
On the first block, $h=0^x$.  On the second block,
\[
        \mex\{0,1\}=2,
\]
so the block is $2^x$.  On the third block,
\[
        \mex\{2,0\}=1,
\]
so the block is $1^x$.

Thereafter, if $j\ge3$ is odd, the preceding block has value $1$
and $g(n)=1$, so
\[
        h(n)=\mex\{1\}=0.
\]
If $j\ge4$ is even, the preceding block has value $0$ and $g(n)=0$,
so
\[
        h(n)=\mex\{0\}=1.
\]
This proves \eqref{eq:h-before-y}.
\end{proof}

We now determine the complete nonperiodic prefix of the pass-available
sequence.

\subsection{The case $q=1$}

Here
\[
        y=x+r,\qquad 0<r<x.
\]

\begin{proposition}
\label{prop:q1}
If $q=1$, then
\begin{equation}
        h[0,3x+r-1]=0^x2^x1^x3^r.
\label{eq:prefix-q1}
\end{equation}
\end{proposition}

\begin{proof}
The first block is $0^x$ by \eqref{eq:h-terminal}.

Consider the second block, $n=x+t$, $0\le t<x$.

If $t<r$, then $n<y$, so Lemma~\ref{lem:pre-y} gives $h(n)=2$.

If $t\ge r$, the $y$-move is legal and
\[
        h(n-x)=0,\qquad h(n-y)=0,\qquad g(n)=1.
\]
Thus
\[
        h(n)=\mex\{0,0,1\}=2.
\]
Hence the entire second block is $2^x$.

Now let $n=2x+t$.  Then $h(n-x)=2$.

For $t<r$,
\[
        n-y=x+t-r<x,
\]
so $h(n-y)=0$, while Theorem~\ref{thm:ordinary} gives $g(n)=2$.
Therefore
\[
        h(n)=\mex\{2,0,2\}=1.
\]

For $t\ge r$, the position $n-y$ lies in the $2^x$ block, so
$h(n-y)=2$.  The ordinary period has wrapped to its initial $0^x$
block, hence $g(n)=0$.  Thus
\[
        h(n)=\mex\{2,2,0\}=1.
\]
The third block is therefore $1^x$.

Finally let
\[
        n=3x+t,\qquad0\le t<r.
\]
Then
\[
        h(n-x)=1,
\]
and
\[
        n-y=2x+t-r\in[x,2x),
\]
so $h(n-y)=2$.

The ordinary period is $0^x1^x2^r$ of length $2x+r$.
Since
\[
        n-(2x+r)=x+t-r\in[0,x),
\]
we have $g(n)=0$.  Hence
\[
        h(n)=\mex\{1,2,0\}=3.
\]
This proves \eqref{eq:prefix-q1}.
\end{proof}

\subsection{Even quotient}

Let
\[
        q=2m,\qquad m\ge1.
\]

\begin{proposition}[Even quotient]
\label{prop:even-prefix}
If $q=2m$, then:

\begin{enumerate}
\item for $r=0$,
\begin{equation}
h[0,(q+2)x-1]
=
0^x2^x1^x(0^x1^x)^{m-1}3^x;
\label{eq:prefix-even-r0}
\end{equation}

\item for $0<r<x$,
\begin{equation}
h[0,(q+2)x-1]
=
0^x2^x1^x(0^x1^x)^{m-1}2^r3^{x-r}.
\label{eq:prefix-even-rpos}
\end{equation}
\end{enumerate}
\end{proposition}

\begin{proof}
Lemma~\ref{lem:pre-y} determines all complete blocks strictly before
the block containing $y$.  It remains to complete the boundary blocks.

First let
\[
        n=qx+t,\qquad0\le t<x.
\]

If $t<r$, then $n<y$.  Since $q$ is even,
Lemma~\ref{lem:pre-y} gives $h(n)=1$.

Now suppose $t\ge r$.  Then
\[
        n-y=t-r<x,
\]
so $h(n-y)=0$.
By \eqref{eq:ordinary-even},
\[
        g(n)=2.
\]

For $q=2$, the $x$-predecessor lies in the exceptional block $2^x$,
so $h(n-x)=2$, and
\[
        \mex\{2,0,2\}=1.
\]
For $q\ge4$, the $x$-predecessor lies in an odd block of index at least
$3$, so $h(n-x)=0$, and
\[
        \mex\{0,0,2\}=1.
\]
Thus in every even-$q$ case,
\begin{equation}
        h[qx,(q+1)x-1]=1^x.
\label{eq:q-even-completed}
\end{equation}

Next let
\[
        n=(q+1)x+t.
\]
Then $h(n-x)=1$.

If $t<r$, then
\[
        n-y=x+t-r\in[0,x),
\]
so $h(n-y)=0$.  From \eqref{eq:ordinary-even}, $g(n)=1$.
Therefore
\[
        h(n)=\mex\{1,0,1\}=2.
\]

If $t\ge r$, then
\[
        n-y=x+t-r\in[x,2x),
\]
so $h(n-y)=2$.
At $t=r$ the ordinary sequence has completed one period of length
$x+y=(q+1)x+r$, and therefore $g(n)=0$ for $t\ge r$.
Thus
\[
        h(n)=\mex\{1,2,0\}=3.
\]

Consequently the final boundary block is
\[
        2^r3^{x-r}.
\]
For $r=0$ this is $3^x$, giving
\eqref{eq:prefix-even-r0}; for $r>0$ it gives
\eqref{eq:prefix-even-rpos}.
\end{proof}

\subsection{Odd quotient}

Let
\[
        q=2m+1.
\]

\begin{proposition}[Odd quotient, $r=0$]
\label{prop:odd-r0-prefix}
If $q=2m+1\ge3$ and $r=0$, then
\begin{equation}
h[0,(q+1)x-1]
=
0^x2^x1^x(0^x1^x)^{m-1}2^x.
\label{eq:prefix-odd-r0}
\end{equation}
\end{proposition}

\begin{proof}
Lemma~\ref{lem:pre-y} determines the sequence up to $n=y-1$.
Let
\[
        n=qx+t,\qquad0\le t<x.
\]
The preceding block has value $1$, so
\[
        h(n-x)=1.
\]
Since $n-y=t$,
\[
        h(n-y)=0.
\]
By \eqref{eq:ordinary-odd-multiple}, the $q$th ordinary block,
with $q$ odd, has value $1$.  Hence
\[
        h(n)=\mex\{1,0,1\}=2.
\]
This proves \eqref{eq:prefix-odd-r0}.
\end{proof}

\begin{proposition}[Odd quotient, $r>0$]
\label{prop:odd-rpos-prefix}
If $q=2m+1\ge3$ and $0<r<x$, then
\begin{equation}
\begin{split}
h[0,(q+2)x+r-1]
={}&
0^x2^x1^x
(0^x1^x)^{m-1}\\
&{}\cdot
0^r2^{x-r}1^x3^r.
\end{split}
\label{eq:prefix-odd-rpos}
\end{equation}
\end{proposition}

\begin{proof}
Again Lemma~\ref{lem:pre-y} determines the full blocks before $y$.

First let
\[
        n=qx+t.
\]
For $t<r$ we have $n<y$.  Since $q\ge3$ is odd,
Lemma~\ref{lem:pre-y} yields $h(n)=0$.

For $t\ge r$,
\[
        h(n-x)=1,\qquad h(n-y)=0.
\]
From \eqref{eq:ordinary-odd-r},
\[
        g(n)=1.
\]
Therefore
\[
        h(n)=\mex\{1,0,1\}=2.
\]
Hence
\[
        h[qx,(q+1)x-1]=0^r2^{x-r}.
\]

Now put
\[
        n=(q+1)x+t.
\]
If $t<r$, then
\[
        h(n-x)=0,\qquad h(n-y)=0,\qquad g(n)=2,
\]
so
\[
        h(n)=\mex\{0,0,2\}=1.
\]
If $t\ge r$, then
\[
        h(n-x)=2,\qquad h(n-y)=2.
\]
The ordinary word has wrapped at $t=r$, so $g(n)=0$.
Thus
\[
        h(n)=\mex\{2,2,0\}=1.
\]
Therefore
\[
        h[(q+1)x,(q+2)x-1]=1^x.
\]

Finally take
\[
        n=(q+2)x+t,\qquad0\le t<r.
\]
Then $h(n-x)=1$ and
\[
        n-y=2x+t-r\in[x,2x),
\]
so $h(n-y)=2$.

The ordinary period length is
\[
        P=x+y=(q+1)x+r.
\]
Since
\[
        n-P=x+t-r\in[0,x),
\]
we have $g(n)=0$.  Hence
\[
        h(n)=\mex\{1,2,0\}=3.
\]
This gives \eqref{eq:prefix-odd-rpos}.
\end{proof}

The preceding propositions give the following natural index.\\

\indent
Define
\begin{equation}
N=N(x,y):=
\begin{cases}
(q+1)x,
    & r=0,\ q\text{ odd},\\[2mm]
(q+2)x,
    & r=0,\ q\text{ even},\\[2mm]
(q+2)x+r,
    & r>0,\ q\text{ odd},\\[2mm]
(q+2)x,
    & r>0,\ q\text{ even}.
\end{cases}
\label{eq:N}
\end{equation}

\noindent
This is equivalent to 
\[
N=
\begin{cases}
y+x,&r=0,\ q\text{ odd},\\
y+2x,&r=0,\ q\text{ even},\\
y+2x,&r>0,\ q\text{ odd},\\
y+2x-r,&r>0,\ q\text{ even}.
\end{cases}
\]

We next prove that this is exactly where the shifted ordinary tail begins.

\begin{lemma}
\label{lem:mex-stability}
Let
\[
        d=\mex\{a,b\}.
\]
If $c\ne d$, then
\[
        \mex\{a,b,c\}=d.
\]
\end{lemma}

\begin{proof}
Every nonnegative integer less than $d$ already occurs among $a,b$,
while $d$ does not.  Adding an element $c\ne d$ preserves both facts.
\end{proof}

\begin{lemma}[Initial tail block]
\label{lem:first-tail-block}
For
\[
        N\le n<N+x,
\]
we have 
\[
        h(n)=g(n+y).
\]
\end{lemma}

\begin{proof}
Write $n=N+t$, $0\le t<x$.  Reading the three values in
\eqref{eq:h-rec} from the already-proved prefix and from
Theorem~\ref{thm:ordinary} gives the following complete table.

\begin{center}
\begin{tabular}{llll}
\toprule
Case & range of $t$
& $(h(n-x),h(n-y),g(n))$
& $h(n)=g(n+y)$\\
\midrule
$q$ odd, $r>0$
  & $0\le t<x-r$ & $(1,1,1)$ & $0$\\
$q$ odd, $r>0$
  & $x-r\le t<x$ & $(3,1,1)$ & $0$\\
$q$ even, $r>0$
  & $0\le t<r$ & $(2,2,0)$ & $1$\\
$q$ even, $r>0$
  & $r\le t<x$ & $(3,1,1)$ & $0$\\
$q$ even, $r=0$
  & $0\le t<x$ & $(3,1,1)$ & $0$\\
$q$ odd, $r=0$
  & $0\le t<x$ & $(2,2,0)$ & $1$\\
\bottomrule
\end{tabular}
\end{center}

For example,
\[
        \mex\{1,1,1\}=0,\qquad
        \mex\{3,1,1\}=0,\qquad
        \mex\{2,2,0\}=1.
\]
The value in the final column is independently the corresponding value
of the explicit ordinary word at $n+y$.  Thus every possible case gives
$h(n)=g(n+y)$.
\end{proof}

For $N+x\le n<N+y$, the $x$-predecessor is already in the desired
tail but the $y$-predecessor still lies in the transient.  This is the
only genuine obstruction to a direct induction.

Set
\[
        z=n-y.
\]
Then
\[
        N+x-y\le z<N.
\]

The key fact is that the transient $y$-predecessor never equals the
desired mex.

\begin{lemma}
\label{lem:avoidance}
For every
\[
        N+x-y\le z<N,
\]
we have 
\begin{equation}
        h(z)\ne g(z+2y).
\label{eq:avoidance}
\end{equation}
\end{lemma}

\begin{proof}
We compare the explicit transient word with the corresponding ordinary
word.  A pair $(a,b)$ below denotes
\[
        (h(z),g(z+2y)).
\]
Every displayed pair has unequal entries.

\medskip
\noindent
\textbf{Case 1: $q=1$, $r>0$.}

Here $N=3x+r$ and $N+x-y=3x$.  The relevant portion of the transient
is $3^r$.  Reduction modulo the ordinary period $2x+r$ gives
$g(z+2y)=1$ throughout the interval.  Hence the pair word is
\[
        (3,1)^r.
\]

\medskip
\noindent
\textbf{Case 2: $q=2m$ even, $r=0$.}

The ordinary period is $(0^x1^x)^m2^x$ of length
\[
        P=(q+1)x.
\]
Since
\[
        2y=2P-2x,
\]
periodicity gives
\[
        g(z+2y)=g(z-2x),
\]
where the latter is interpreted using the periodic extension.
Aligning the words yields
\[
\Bigl((0,1)^x(1,0)^x\Bigr)^{m-1}(3,1)^x.
\]
All pairs are unequal.

\medskip
\noindent
\textbf{Case 3: $q=2m$ even, $r>0$.}

Again $P=x+y$ and
\[
        2y=2P-2x.
\]
The relevant alignment is
\[
(1,0)^r
\Bigl((0,1)^x(1,0)^x\Bigr)^{m-1}
(2,1)^r(3,1)^{x-r}.
\]
Every pair has different entries.

\medskip
\noindent
\textbf{Case 4: $q=2m+1\ge3$ odd, $r=0$.}

The ordinary period is $2x$, and $2y$ is a multiple of $2x$.
Therefore
\[
        g(z+2y)=g(z).
\]
The pair word is
\[
(1,0)^x
\Bigl((0,1)^x(1,0)^x\Bigr)^{m-1}
(2,1)^x.
\]
Again no entries agree.

\medskip
\noindent
\textbf{Case 5: $q=2m+1\ge3$ odd, $r>0$.}

Here $P=x+y$ and again
\[
        2y=2P-2x.
\]
The aligned pair word is
\[
\Bigl((0,1)^x(1,0)^x\Bigr)^{m-1}
(0,1)^r(2,1)^{x-r}
(1,0)^x(3,1)^r.
\]
Every pair is unequal.

Thus \eqref{eq:avoidance} holds in every case.
\end{proof}

\begin{lemma}[Bridge lemma]
\label{lem:bridge}
For every
\[
        N\le n<N+y,
\]
we have 
\begin{equation}
        h(n)=g(n+y).
\label{eq:bridge}
\end{equation}
\end{lemma}

\begin{proof}
Lemma~\ref{lem:first-tail-block} proves the claim for
\[
        N\le n<N+x.
\]

Proceed by induction on $n$ through the remaining bridge
\[
        N+x\le n<N+y.
\]
Assume the equality has been proved for all previous indices at least
$N$.  Since $n-x\ge N$,
\[
        h(n-x)=g(n+y-x).
\]
Set
\[
        a=g(n+y-x),\qquad b=g(n).
\]
The ordinary recurrence at $n+y$ gives
\[
        g(n+y)=\mex\{a,b\}.
\]

Now set
\[
        c=h(n-y).
\]
Since
\[
        N+x-y\le n-y<N,
\]
Lemma~\ref{lem:avoidance} gives
\[
        c\ne g(n+y).
\]
By Lemma~\ref{lem:mex-stability},
\[
        \mex\{a,b,c\}=\mex\{a,b\}=g(n+y).
\]
Using \eqref{eq:h-rec},
\[
        h(n)=g(n+y).
\]
This completes the induction across the bridge.
\end{proof}

\begin{theorem}
\label{thm:eventual-translation}
For every
\[
        n\ge N,
\]
we have 
\begin{equation}\label{eq:eventual-translation}
        h(n)=g(n+y).
\end{equation}
\end{theorem}

\begin{proof}
Lemma~\ref{lem:bridge} establishes the claim through $N+y-1$.

Now let $n\ge N+y$, and suppose the claim holds at every smaller index
at least $N$.  Then
\[
        n-y\ge N
\]
and, because $x<y$,
\[
        n-x>n-y\ge N.
\]
Hence
\[
        h(n-x)=g(n+y-x),
        \qquad
        h(n-y)=g(n).
\]
Substituting into \eqref{eq:h-rec},
\[
\begin{aligned}
h(n)
&=\mex\{g(n+y-x),g(n),g(n)\}\\
&=\mex\{g(n+y-x),g(n)\}\\
&=g(n+y),
\end{aligned}
\]
where the last equality is precisely the ordinary recurrence
\eqref{eq:g-rec} evaluated at $n+y$.
Strong induction completes the proof.
\end{proof}

\section{The Grundy-value sequence of $S^\ast = \{x, y\}$} \label{Main Result}

\begin{theorem}
Let $S^\ast=\{x,y\}$, $0<x<y$ and let 
\[
        y=qx+r,\qquad0\le r<x.
\]

Then the Grundy-value sequence $h$ of $S^\ast$ is as follows.\\

\begin{enumerate}

\item If $q=1$ $($so $0<r<x)$, then
\begin{equation}\label{eq:full-q1}
h=
0^x2^x1^x3^r
\left(0^x1^x2^r\right)^\infty.
\end{equation}

\item If $q=2m$ is even and $r=0$, then
\begin{equation}\label{eq:full-even-r0}
h=
0^x2^x1^x(0^x1^x)^{m-1}3^x
\left((0^x1^x)^m2^x\right)^\infty.
\end{equation}

\item If $q=2m$ is even and $0<r<x$, then
\begin{equation}\label{eq:full-even-rpos}
h=
0^x2^x1^x(0^x1^x)^{m-1}2^r3^{x-r} \left(
1^r(0^x1^x)^m0^r2^{x-r}
\right)^\infty.
\end{equation}

\item If $q=2m+1\ge3$ is odd and $r=0$, then
\begin{equation}\label{eq:full-odd-r0}
h=
0^x2^x1^x(0^x1^x)^{m-1}2^x
\left(1^x0^x\right)^\infty.
\end{equation}

\item If $q=2m+1$ is odd and $0<r<x$, then
\begin{equation}\label{eq:full-odd-rpos}
h=
0^x2^x1^x(0^x1^x)^{m-1}
0^r2^{x-r}1^x3^r \left((0^x1^x)^{m+1}2^r\right)^\infty.
\end{equation}

\end{enumerate}

Equivalently, in all five cases,
\[
        h(n)=g(n+y)\qquad(n\ge N),
\]
where $N$ is given by \eqref{eq:N}.
\end{theorem}

\begin{proof}
The finite prefixes are
Propositions~\ref{prop:q1}, \ref{prop:even-prefix},
\ref{prop:odd-r0-prefix}, and \ref{prop:odd-rpos-prefix}.
Theorem~\ref{thm:eventual-translation} identifies the remainder of the
sequence as the $y$-shifted ordinary sequence.

It remains only to note where that shift enters the ordinary period.

For $q=1$,
\[
        N+y=2(2x+r),
\]
so the tail starts at the beginning of the ordinary period.

For even $q$ and $r=0$,
\[
        N+y=2(x+y),
\]
again starting at the beginning of the period.

For even $q$ and $r>0$,
\[
        N+y=2(x+y)-r,
\]
so the tail starts with the final $1^r$ block of
\eqref{eq:ordinary-even}, then wraps around.  This gives
\eqref{eq:full-even-rpos}.

For odd $q$ and $r=0$, reduction of $N+y$ modulo $2x$ gives $x$,
so the tail starts at the $1^x$ half of the ordinary word.

For odd $q$ and $r>0$,
\[
        N+y=2(x+y),
\]
so the tail begins at the start of the ordinary period.
\end{proof}

\section{Concluding remarks}\label{Remarks}

\begin{theorem}
\label{thm:period}
The eventual least period of the pass-available sequence is
\begin{equation}\label{eq:period}
P=
\begin{cases}
2x,& y=(2k+1)x\text{ for some }k\ge0,\\[2mm]
x+y,&\text{otherwise}.
\end{cases}
\end{equation}
\end{theorem}

\vspace{10pt}

\begin{proof}
By Theorem~\ref{thm:eventual-translation}, the tail of $h$ is a shift
of the ordinary sequence $g$.  Hence every period of $g$ is an eventual
period of $h$.

Conversely, suppose the tail of $h$, and therefore a tail of $g$, had
a positive period $d$ smaller than the ordinary least period.
Because $g$ is purely periodic from $0$, with ordinary period $P$,
for any $n$ choose $k$ so large that both $n+kP$ and $n+kP+d$ lie
beyond the assumed eventual threshold.  Then
\[
\begin{aligned}
g(n+d)
&=g(n+kP+d)\\
&=g(n+kP)\\
&=g(n).
\end{aligned}
\]
Thus $d$ would be a global period of $g$, contradicting the exact least
period in Theorem~\ref{thm:ordinary}.

Therefore the eventual least period of $h$ equals the ordinary least
period, which is precisely \eqref{eq:period}.
\end{proof}

The index $N$ is not merely an upper bound.

\begin{theorem}
\label{thm:least-preperiod}
The number $N$ in \eqref{eq:N} is the least index from which the
pass-available sequence is periodic with its eventual least period $P$.
\end{theorem}

\begin{proof}
Periodicity with period $P$ holds for all $n\ge N$ by
Theorems~\ref{thm:eventual-translation} and \ref{thm:period}.

It remains to show that periodicity cannot begin at $N-1$.
Since $N-1+P\ge N$, Theorem~\ref{thm:eventual-translation} gives
\[
        h(N-1+P)
        =
        g(N-1+P+y)
        =
        g(N-1+y),
\]
where the last equality uses the ordinary period $P$.

We compare this value with the final value of the explicit transient.

If $r>0$, then in every case the prefix ends in $3$, so
\[
        h(N-1)=3.
\]
For $q=1$ or odd $q$,
the index $N-1+y$ is congruent to the final position of an ordinary
period, where the value is $2$.
For even $q$ and $r>0$, reducing by one ordinary period puts
$N-1+y$ at the final position of the $2^{x-r}$ block, again giving
\[
        g(N-1+y)=2.
\]
Thus $3\ne2$.

If $r=0$ and $q$ is even, the prefix again ends in $3$, while
$g(N-1+y)=2$.

Finally suppose $r=0$ and $q$ is odd.
Then the prefix ends in $2$, whereas the ordinary sequence is
$(0^x1^x)^\infty$.  Since
\[
        N-1+y=(2q+1)x-1
\]
is congruent modulo $2x$ to $x-1$, it lies in a $0^x$ block.
Hence
\[
        h(N-1)=2
        \ne
        0=g(N-1+y)
        =h(N-1+P).
\]

Therefore the period relation fails at $N-1$ in every case.
Hence $N$ is the least preperiod.
\end{proof}

The single pass changes the structure of the two-element subtraction
game in a precise way. First, it destroys pure periodicity at the origin in general.  The
finite transient can contain the new nimber $3$. Second, the perturbation is finite.  After the exact threshold $N$,
the game forgets the pass in the particularly strong sense
\[
        h(n)=g(n+y).
\]
Thus the pass-available sequence is not merely eventually periodic:
its tail is literally a translate of the ordinary nim-sequence. Third, the single pass changes neither the eventual least period nor the underlying ordinary periodic word, except for a cyclic rotation. It therefore acts as a finite-state transient perturbation of the
ordinary two-element subtraction game. Finally, the quotient--remainder decomposition
\[
        y=qx+r
\]
has a structural interpretation.  The parity of $q$ determines which
alternating $x$-block is encountered when the $y$-move first becomes
active, while $r$ determines where that boundary cuts through an
$x$-block.  These two data completely determine the transient.

\section{Tool and computational resource disclosure}\label{Disclosure}

The author supports the Leiden Declaration on Artificial Intelligence and Mathematics. GPT-5.6 Sol was used for the straightforward (but tedious) calculations in this paper.

\end{document}